\documentclass[11pt]{article}

\usepackage[english]{babel}

\usepackage[letterpaper,top=2cm,bottom=2cm,left=3cm,right=3cm,marginparwidth=1.75cm]{geometry}

\pdfoutput=1 
\usepackage{graphicx}      
\usepackage{natbib}        
\usepackage{amsmath}
\usepackage{tikz}
\usepackage{amsfonts}
\usepackage{dsfont}
\usepackage{blkarray}
\usepackage{mathtools}
\mathtoolsset{showonlyrefs}
\usepackage[utf8]{inputenc}
\usepackage[T1]{fontenc}

\usepackage{amsmath}
\usepackage{graphicx}
\usepackage[colorlinks=true, allcolors=blue]{hyperref}
\usepackage[utf8]{inputenc}
\usepackage{url}
\usepackage{hyperref}
\usepackage{amsmath}
\usepackage{amssymb}
\usepackage{accents}
\usepackage[english]{babel}
\usepackage{amsthm}
\usepackage{geometry,mathtools}
\usepackage{esvect}
\usepackage{dsfont}
\usepackage{caption}
\usepackage{subcaption}
\usepackage{tikz}
\usetikzlibrary{automata,arrows,positioning,calc}
\usepackage{blkarray}
\usepackage{authblk}

\usepackage{algorithm}
\usepackage{algorithmic}

\usepackage{titlesec} 
\usepackage{afterpage} 
\usepackage{lineno}

\usepackage{babel}

\usepackage{comment}
\usepackage{changepage}

\newtheorem{thm}{Theorem}

\newtheorem{defn}{Definition}[section]
\newtheorem{lem}{Lemma}

\newtheorem{rem}{Remark}
\newtheorem{exmp}{Example}[section]

\title{Constructing Stochastic Matrices for Weighted Averaging in Gossip Networks}
    \author{Erkan Bayram\quad  Mohamed-Ali Belabbas\\
	\normalsize Coordinated Science Laboratory\\
	\normalsize University of Illinois Urbana-Champaign, Urbana, IL, USA\\
	\normalsize  \emph{(ebayram2,belabbas)@illinois.edu}}

\begin{document}
\date{}
\maketitle





                                                  

\begin{abstract}                          
The convergence of the gossip process has been extensively studied; however, algorithms that generate a set of stochastic matrices, the infinite product of which converges to a rank-one matrix determined by a given weight vector, have been less explored. In this work, we propose an algorithm for constructing (local) stochastic matrices based on a given gossip network topology and a set of weights for averaging across different consensus clusters, ensuring that the gossip process converges to a finite limit set.


\end{abstract}

\section{Introduction}

Distributed control systems are fundamental in modern computing, enabling scalable and efficient data processing across multiple agents~\citep{bayram2024ageCoded}. Distributed averaging plays a key role in applications such as distributed optimization and decentralized learning, where each agent in a network contributes to consensus value based on the agreed-upon weights. One essential protocol for distributed averaging is the gossip process, in which two adjacent agents in a network communicate and update their states based on the matrix weight $A_e$ (a row-stochastic matrix) associated with the edge $e$ connecting them. These two agents form a gossiping pair, and the overall process is referred to as a {\em weighted gossip process}~\citep{bayram2024age,chen2022gossip}. This process is closely related to a non-homogeneous Markov process, where each edge has a different weight matrix.

This formulation reduces the problem to analyzing the infinite products of stochastic matrices taken from a finite set. When agents have vector-valued states, the ensemble of agents, or more generally, the set of the entries of their state vectors, can be partitioned into subvectors such that all the subvectors in the same partition reach consensus. This process is called {\em multiple consensus} (i.e. class-ergodicity). The elements of this partition are referred to as a {\em consensus clusters} ~\citep{bolouki2015consensus,touri2012approximations}. As a shorthand, we refer to the entries of the state-vector of an agent as the {\em entries} of the agent.

While prior works have primarily focused on conditions under which the product of stochastic matrices converges to a rank-one matrix (or principal block matrices with rank one, i.e. multiple-consensus), a less explored but critical problem is the realization of gossip matrices for a given gossip network. Specifically, for a given network topology (i.e., a graph) and a set of weights for averaging, constructing a meaningful set of stochastic matrices for gossip process, the infinite product of which converges to finite limit set remains an open problem. In this work, we propose an algorithm to construct (i.e., realize) the set of stochastic matrices $A_e$ that govern the gossip process such that the infinite product of these matrices converges to multiple consensus for a given weight vector (which sets the averaging weights at consensus) and a given partition of states that constitute the consensus clusters.

The existence and convergence of such products have been extensively studied in various contexts, including Lyapunov function-based methods \citep{nedic2016convergence,fagnani2008randomized}, consensus  constrained by network topologies \citep{Morse_etal2008ReachingConsensus,  Morse_etAl2008Dynamically,   ren2005consensus}, continuous-time models~\citep{hendrickx2012convergence}, and ergodic theory~\citep{touri2010ergodicity}.

At the same time, the problem of realizing of stochastic matrices for a given weight vector, serving as the left eigenvector of the infinite product, is closely linked to constructing stochastic matrices with a known spectrum \citep{kolmogorov1937markov}. \cite{dmitriev1946characteristic} and {\cite{karpelevich1951characteristic} impose necessary conditions on the spectrum of row-stochastic matrices but do not provide explicit realization algorithms. \cite{johnson2017matricial} offer a realization method, but it is limited to Karpelevic arcs. 

Given this setup, several key questions remain unresolved:
\begin{enumerate}  
    \item Under what topological conditions can two agents belong to the same consensus cluster? More generally, can different entries of distinct agents belong to the same consensus cluster?
    \item Given a graph and a weight vector, what is the way to construct (i.e. realize) a finite set of stochastic matrices $\mathcal{A}=\{A_{e_1},A_{e_2},\cdots,A_{e_\ell}\}$ such that the infinite product of these matrices in some order:
    \begin{align}
        \lim_{k \to \infty} A_k \cdots  A_1 \mbox{ with } A_i \in {\cal A}\\[-1.75em]
    \end{align}
    converges to a limit set with finite cardinality, ensuring the desired consensus clusters and the weights of average?
\end{enumerate}

The topology of the given (communication) graph $G$ inherently restricts the formation of consensus clusters. To address the first question, we introduce the so-called derived graph of $G$ on the elements of the partition (i.e., the set of indices labeling each consensus cluster). This derived graph provides a method to test whether a user-defined partition of the entries of the agents on $G$ forms an admissible set of consensus clusters.




Recent works such as \citep{chen2022gossip} and \citep{bayram2023vector} borrow the notion of holonomy from geometric control. The notion of holonomy in~\cite{bayram2023vector} refers to the structure induced by a set of stochastic matrices acting on a given weight vector around a cycle. In particular, when this set of matrices exhibits finite orbit sets, meaning that iterating the update process results in a cyclic progression through a finite set of weight vectors, this provides a powerful algebraic tool for analyzing the long-term behavior of the system. Both~\cite{chen2022gossip} and \citep{bayram2023vector} established sufficient conditions (introduced as $w$-holonomy), for the convergence of gossip processes to a finite limit set. To address to the second question, for a given gossip network $G$, a partition for the consensus clusters and a set of weights for averaging, we provide an algorithm to construct a set of (local) stochastic matrices such that they are $w$-holonomic for the graph $G$. From the result~\citep{bayram2023vector}, such a set of local stochastic matrices (i.e. $w$-holonomic for $G$) ensures that the gossip process converges to a finite limit set.

{ We summarize our main contributions as follows:}
\begin{itemize} \item We propose a method to test whether a given partition of the entries of state vector on the network forms an admissible set of consensus clusters. 
\item For a given graph, weight vector, and consensus cluster partition, we propose an algorithm to construct local stochastic matrices that govern the gossip process, ensuring convergence to multiple consensus ,where the clusters align with the given partition and the averaging weights match the specified vector.
\item We provide a solution to the open problem of realizing stochastic matrices with a known left eigenvector corresponding to eigenvalue $1$.
\item Our work ensures that the gossip process converges to a limit set with finite cardinality, thus enabling the design of efficient distributed control systems.
\end{itemize}

\textbf{Notation and convention.} We denote by $G=(V,E)$ an undirected graph, with $V=\{v_1,\ldots,v_{|V|}\}$ the node set and $E=\{e_1,\ldots,e_{|E|}\}$ the edge set. The edge linking nodes $v_i$ and $v_j$ is denoted by $(v_i, v_j)$, a self loop is denoted by $(v_i , v_i)$. We call $G$ {\em simple}  if it has no self-loops. Given a sequence of edges $\gamma = e_1 \cdots  e_k$ in $E$, a node $v \in V$ is called \textbf{covered} by $\gamma$  if it is incident to an edge in $\gamma$. A pointed cycle in \( \vec{G} \) is a walk \( v_{i_1} v_{i_2} \cdots v_{i_k} v_{i_1} \) with basepoint \( v_{i_1} \). Let \( \vec{\mathcal{C}} \) be the set of all pointed cycles in \( \vec{G} \). We define cycles as equivalence classes of pointed cycles that visit the same vertices in the same cyclic order and denote this set also by \( \vec{\mathcal{C}} \).



{A vector $p \in \mathbb{R}^n$ is a probability vector if $p_i \geq 0$ and $\sum_{i=1}^{n} p_i = 1$. The set of such vectors, the $(n-1)$-simplex $\Delta^{n-1}$, has interior $\operatorname{int} \Delta^{n-1}$ in the Euclidean topology, where all entries of $p$ are strictly positive.}

\section{Preliminaries}

\textbf{Gossip Process.} Consider an undirected simple graph $G=(V,E)$ with $n$ nodes. Each node represents an agent, and each agents' state is a vector in $\mathbb{R}^m$. We denote the state vector of the agent $i$ at time $t$ by ${x^i(t)}= \left[ x_1^i(t) , x_2^i(t) ,\ldots , x_m^i(t) \right]^\top \in \mathbb{R}^m $, where $x^i_k(t)$ is the $k$th entry of the agent $i$. The state of the system is the concatenation of the agents' states
$$
x(t)= [x^1(t)^\top x^2(t)^\top \cdots x^n(t)^\top]^\top \in \mathbb{R}^{nm}.
$$
The stochastic process we analyze here is described by sequences of edges $\gamma=e_{i_1}\cdots e_{i_t} \cdots$ in $G$ with the convention that {if $e_{i_t}=(v_i,v_j)$, then agents $i$ and $j$ update their states according to a row stochastic matrix ${A_{ij}}$, called {\em local stochastic matrix}. 
Then, the gossip process on edge $e_{i_t}=(v_i,v_j)$ at time $t$ is given by } 
\begin{equation}\label{eqn:gossip_process}
x(t+1) = {A}_{ij} x(t).  
\end{equation}
The matrix $A_{ij}$ is an $nm$-dimensional stochastic matrix such that the rows/columns corresponding to the states of agent $i$ and agent $j$ is the principal submatrix $\tilde{A}_{ij}$ and the rows/columns corresponding to the other agents is the identity matrix. For example, the local stochastic matrix $A_{12}$, which is associated with the edge $(v_1,v_2)$, is given by
\begin{equation}\label{block}
A_{12} = 
\begin{bmatrix}
\tilde{A}_{12} & 0_{2m \times (n-2)m}\\
0_{(n-2)m \times 2m} & I_{(n-2)m \times (n-2)m}
\end{bmatrix}
\end{equation}
 We assume here that $A_{ij} = A_{ji}$. For a simple undirected graph $G$, the directed graph $\vec{G} = (V, \vec{E})$ is obtained by bidirectionalizing the edges of $G$: for each edge $(v_i, v_j)$ in $G$ with $i \neq j$, we add directed edges $v_iv_j$ and $v_jv_i$ in $\vec{G}$. Thus, in $\vec{G}$, each edge is associated with $A_{ij}$ for both directions, meaning the graph $G$ is a matrix-weighted graph. An example of a matrix-weighted graph on 7 nodes is shown in Figure~\ref{fig:butterfly_graph}.

For a finite sequence $\gamma = e_1 \cdots e_k$ of edges in $G$ and for a given pair of integers
$0 \leq s \leq t \leq k$, we define the transition matrix $P_\gamma(t : s)$ for $t \ge s + 1$ as {the left product of local stochastic matrices from $s+1$ to $t$, given by}:
\begin{equation}\label{eqn:state_transition_over_walk}
P_{\gamma}(t : s) := A_{e_t} A_{e_{t-1}}\cdots A_{e_{s+1}}
\end{equation}
We set $P_\gamma(t : s) = I$ for $t \leq s$. Then, we have the following update for the state vector $x$ at $s$:
\begin{equation}\label{eqn:state_transition}
x(t) = P_\gamma(t:s) x(s). 
\end{equation}
We will simply write $P_\gamma$ for $P_{\gamma}(t : s)$ when clear from the context.



\begin{figure}
\centering
\begin{tikzpicture}[scale=0.9]
		\node [circle,fill=black,inner sep=1pt,label=above:{\footnotesize $v_1$}] (b_i) at (3, 0) {};
		
		\node [circle,fill=black,inner sep=1pt,label=above:{\footnotesize $v_2$}] (b_j) at (1, 0.75) {};

		\node [circle,fill=black,inner sep=1pt,label=below:{\footnotesize $v_3$}] (b_l) at (1, -0.75) {};
        \node [circle,fill=black,inner sep=1pt,label=above:{\footnotesize $v_4$}] (b_g) at (5, 0.75) {};
        \node [circle,fill=black,inner sep=1pt,label=below:{\footnotesize $v_5$}] (b_q) at (5, -0.75) {};
	    
        \node [circle,fill=black,inner sep=1pt,label=below:{\footnotesize $v_6$}] (b_y) at (7, -0.75) {};
        \node [circle,fill=black,inner sep=1pt,label=above:{\footnotesize $v_7$}] (b_u) at (7, 0.75) {};

   \path[draw,every node/.style={sloped,anchor=south,auto=false},shorten >=2pt,shorten <=2pt]
		(b_i) edge[<->] node {$A_{12}$}  (b_j)

        (b_j) edge[ <->] node {$A_{23}$} (b_l)
        (b_l) edge[<->] node {$A_{31}$} (b_i)	

        (b_i) edge[<->] node {$A_{14}$} (b_g)
        (b_g) edge[<->] node {$A_{45}$} (b_q)
        (b_i) edge[<->] node {$A_{15}$} (b_q)

        (b_y) edge[<->] node {$A_{56}$} (b_q)
        (b_y) edge[<->] node {$A_{76}$}(b_u)
        (b_u) edge[<->] node {$A_{57}$} (b_q)

		 ;
\end{tikzpicture}
\caption{The graph $\vec{G}$ } 
\label{fig:butterfly_graph}
\end{figure}
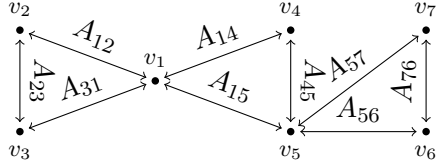

\textbf{Holonomy in the Network.} We employ the concept of holonomy—a powerful algebraic tool that captures the cyclic behavior of weight vector evolution in gossip processes~\citep{chen2022gossip}. 
For a given cycle $C$ and a weight vector $w$, the process which evolves around the cycle $C$ is said to be $w$-holonomic for $C$ if the weight vector $w$ does not change after completing the cycle (i.e. multiplying by $P_C$) or it does {\em change after completing the cycle} once but comes back to the initial value after completing the cycle finite $k$-times. To be more precise:

\begin{defn}\citep[Holonomic Stochastic Matrices]{bayram2023vector}\label{def:non-trivial} Let $C$ be a cycle in $\vec{G}$ of length greater than $2$ and $w^\top \in \operatorname{int}\Delta^{nm-1}$ be a weight vector. The {\em $w$-order} of $C$ is defined as 
$$
\operatorname{ord}_wC := \min \{k\geq 1: {w}={w}( P_C)^k\},
$$
and $\operatorname{ord}_w C=0$ if the set is empty. The local stochastic matrices $A_{e}$,  $e \in C$,  are said to be \textbf{$w$-holonomic for $C$} if there exists a weight vector $w$ such that $\operatorname{ord}_wC$ is finite and non-zero.
\end{defn}
The local stochastic matrices $A_{e}$ are {\em $w$-holonomic for $G$} if there exists a {\em common} weight vector $w$ so that the $A_e$'s are $w$-holonomic for all $C \in \vec{\mathcal{C}}$ of length greater than $2$. 

For a given graph $G = (V, E)$ (satisfying two additional topological conditions: $2$-edge connected and simple), Theorem 1 in \citep{bayram2023vector} shows that if the set of local stochastic matrices $\{A_e , e \in E\}$ is $w$-holonomic for $G$, the infinite product of stochastic matrices associated with the edges in $E$ converges to a limit set with finite cardinality, which depends on the weight vector $w$, provided that an allowable sequence of updates is followed. Note that this result is strong, as it establishes the existence of infinitely many allowable sequences, each of which can be followed in a decentralized manner.




Let $w=[\alpha^1_1,\alpha^1_2,\cdots,\alpha^1_m,\alpha^2_1,\cdots,\alpha^n_{m-1},\alpha^n_{m}]^\top \in \operatorname{int}\Delta^{nm-1}$ be the given weight vector (by the user) for averaging, where $\alpha^i_k$ denotes the weight assigned to the $k$th entry of agent $i$ at consensus. Let $\{1,\ldots, nm\}$ be a set of indices that label the entries of the weight vector $w$. We say that index $k$ in the index set belongs to an agent, namely agent $i$, if 
$$ (i-1)m +1 \leq k \leq  im.
$$
In this work, we consider that a set $\pi$ is given by the user, $\pi \subset 2^{\{1,\ldots,nm\}}$ is a partition of $\{1,\ldots,nm\}$. We denote by $\pi_a, a=0,\ldots, \ell$, the elements of $\pi$ where $\ell+1$ is the cardinality of $\pi$, also denoted by $|\pi|$. We define the subvector of $w$ induced by the element $\pi_a$ of partition $\pi$ as the vector of dimension $|\pi_a|$, consisting of the entries from $w$ indexed by $\pi_a$, ordered in ascending order. Each subvector is normalized by dividing its entries by their sum, ensuring it is a probability vector in $\operatorname{int}\Delta^{|\pi_a|-1}$. Each element $\pi_a$ of the partition corresponds to a consensus cluster, and the consensus value at the limit is the weighted average of the initial state vector, with weights provided by the subvector of $w$ induced by $\pi_a$.

\section{Main Results}


\textbf{Admissible Partition of Index Set on $G$} We need to test whether a given partition $\pi$ of an index set ensures that the proposed consensus clusters can be realized and that the specified indices can contribute to their respective clusters. Therefore, we introduce the concept of a derived graph to formalize how indices within the same partition remain connected. We then define admissible partition of the index set based on this concept.
\begin{defn}[Derived Graph $\mathcal{D}^G(\cdot)$ Function]\label{def:derived_graph_w}
Let $G = (V, E)$ be a {matrix-weighted} graph on $n$ nodes. Let $\pi=\{\pi_a\}_{a=0}^\ell$ be a partition of index set. For a given element $\pi_a$ of the partition of the index set $\pi$, the \textit{derived graph of $G$ on the element $\pi_a$ of the partition}, denoted by $\mathcal{D}^G(\pi_a) = (N_{\pi_a}, {E}_{\pi_a})$, is an undirected graph, possibly with multi-edges, where $N_{\pi_a} = \pi_a$. There is an edge $e_{kl} \in {E}_{\pi_a}$ between nodes $k$ and $l$ if one of the following holds:  
\begin{enumerate}
    \item Indices $k$ and $l$ belong to the same agent. 
    \item  There exist $(v_i,v_j) \in E$ where $i$ and $j$ are the agents to which the indices $k$ and $j$ belong, respectively.
\end{enumerate}
\end{defn}


If the derived graph of $G$ on all elements $\pi_a$ of the partition $\pi$, denoted by $\mathcal{D}^G(\pi_a)$, is strongly connected for all $a \in \{0, 1, \dots, \ell\}$, then we call the partition ${\pi}$ an {\em admissible partition on $G$}.

By using the derived graph, we can verify that there exists a path in $G$ connecting two agents such that every agent on the path has at least two entries whose indices belong to the same proposed consensus cluster. The existence of such a path ensures that information can propagate within the cluster. This approach eliminates the need for direct gossiping between every pair of agents within the cluster. Then, we provide the following example:




 

\begin{exmp}\label{exmp:derived}  Assume that $n=4$ and $m=2$. Let $G=(V,E)$ be the graph such that $V=\{v_1,v_2,v_3,v_4\}$ and $E=\{(v_1,v_2),(v_2,v_3),(v_3,v_4),(v_4,,v_1)\}$. In words, $G$ is a cyclic graph on $4$ nodes. Consider the partition $\pi = \{{\pi_0} = {\{1,3\}} , {\pi_1} = {\{2,4,5,7\}},{\pi_2}={\{6,8\}}  \}$ is given.  
Consider the derived graph $\mathcal{D}^G(\pi_1)$. It consists of four nodes corresponding to indices $\{2,4,5,7\}$. Index $2$ belongs to agent $1$, index $4$ belongs to agent $2$, index $5$ belongs to agent $3$, and index $7$ belongs to agent $4$. Since the graph $G$ contains edges connecting these agents, such as $(v_1,v_2)$ leading to $e_{24} \in E_{\pi_1}$, $(v_2,v_3)$ leading to $e_{45} \in E_{\pi_1}$, $(v_3,v_4)$ leading to $e_{57} \in E_{\pi_1}$, and $(v_4,v_1)$ leading to $e_{27} \in E_{\pi_1}$. It ensures that the derived graph $\mathcal{D}^G(\pi_1)$ is strongly connected (see Figure~\ref{fig:derived_combined}). The derived graph $\mathcal{D}^G(\pi_a)$ is strongly connected for all $a=0,1,2$, which implies that the partition ${\pi}$ is an {\em admissible partition on $G$}



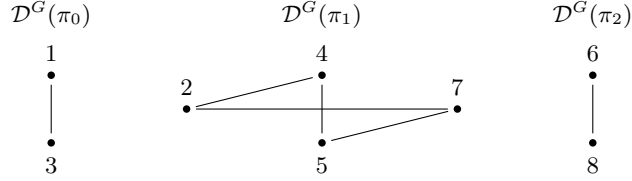
\begin{figure}
\centering
\begin{tikzpicture}[scale=0.9]
    \begin{scope}
        \node at (1,1.4) {\footnotesize $\mathcal{D}^G(\pi_0)$};
        
        \node [circle,fill=black,inner sep=1pt,label=above:{\footnotesize $1$}] (v1) at (1, 0.5) {};
        \node [circle,fill=black,inner sep=1pt,label=below:{\footnotesize $3$}] (v3) at (1, -0.5) {};

        \path[draw,shorten >=2pt,shorten <=2pt]
            (v1) edge[-] node {} (v3);
    \end{scope}

    \begin{scope}[xshift=3cm]
        \node at (2, 1.4) {\footnotesize $\mathcal{D}^G(\pi_1)$};

        \node [circle,fill=black,inner sep=1pt,label=above:{\footnotesize $2$}] (v2) at (0, 0) {};
        \node [circle,fill=black,inner sep=1pt,label=above:{\footnotesize $4$}] (v4) at (2, 0.5) {};
        \node [circle,fill=black,inner sep=1pt,label=below:{\footnotesize $5$}] (v5) at (2, -0.5) {};
        \node [circle,fill=black,inner sep=1pt,label=above:{\footnotesize $7$}] (v7) at (4, 0.0) {};

        \path[draw,shorten >=2pt,shorten <=2pt]
            (v2) edge[-] node {} (v4)
            (v4) edge[-] node {} (v5)
            (v5) edge[-] node {} (v7)
            (v2) edge[-] node {} (v7);
    \end{scope}

    \begin{scope}[xshift=8cm]
        \node at (1, 1.4) {\footnotesize $\mathcal{D}^G(\pi_2)$};

        \node [circle,fill=black,inner sep=1pt,label=above:{\footnotesize $6$}] (v6) at (1, 0.5) {};
        \node [circle,fill=black,inner sep=1pt,label=below:{\footnotesize $8$}] (v8) at (1, -0.5) {};

        \path[draw,shorten >=2pt,shorten <=2pt]
            (v6) edge[-] node {} (v8);
    \end{scope}
\end{tikzpicture}
\caption{Derived graphs $\mathcal{D}^G(\pi_0)$, $\mathcal{D}^G(\pi_1)$, and $\mathcal{D}^G(\pi_2)$ }
\label{fig:derived_combined}
\end{figure}

\end{exmp}


\textbf{Permutation Blocks for $A_e$.} 
From Definition~\ref{def:non-trivial}, we know that some cycles can have $w$-order greater than $1$. In this section, we present a lemma to show that it is sufficient to put a permutation matrix into the state transition matrix of a given cycle $C$ to have the $w$-order of the cycle $C$ is greater than $1$. Note that having $w$-order of a cycle $C$ greater than $1$ is not a necessary condition for having $w$-holonomy for $G$. Therefore, in this section, we provide a design choice for the user by using permutation matrices. 

A stochastic matrix $A$ has {\em a permutation block for an index set $\pi_a$} if the principal block submatrix of $A$ with rows/columns indexed by $\pi_a$ is a permutation matrix with dimension ${|\pi_a|}$. Note that we allow local stochastic matrices to have entries as $\{0,1\}$. This might lead to have permutation blocks in local stochastic matrices and lead permutation blocks at limit for certain indices (see~\cite[Theorem 1]{bayram2023vector}) and the state vector entries labeled by these indices do not contribute to weighted average at the limit. 

Given this, we ask the user to fill $\pi_0$ with the indices of the entries of the state vector that are desired to be in the permutation block at the limit (this set can be empty, as will be discussed later). For a non-empty $\pi_0$, we denote by $w_0$ the subvector of a given weight vector $w$ induced by $\pi_0$ and by $\tilde{w}$ the remaining subvector of $w$ (i.e. induced by $\cup_{a=1}^\ell \pi_a$). For a given cycle $C$, we denote by $P$ and $M$ the principal block submatrices of $P_C$ whose entries are indexed by $\pi_0$ and $\cup_{a=1}^\ell \pi_a$, respectively. Then, we have the following lemma:
\begin{lem}\label{lem:P_c_block}
If $P$ is a permutation matrix with order ${|\pi_0|}$ such that $w_0 P \neq w_0 $ and $M$ is a stochastic matrix such that $\tilde w M=\tilde w$, that is, it holds that (up to labelling): 
\begin{equation}\label{eqn:factorization_Pc}
P_C = \begin{blockarray}{ccc}
\pi_0 & \cup_{a=1}^\ell \pi_a \\
\begin{block}{[cc]c}
P & \mathbf{0} & \pi_0 \\
\mathbf{0} & M & \cup_{a=1}^\ell \pi_a \\
\end{block}
\end{blockarray},
\end{equation}
then the $w$-order of the cycle $C$ is strictly greater than $1$. 
\end{lem}


\begin{rem}
The block structure given in the lemma is also a necessary condition for having $\operatorname{ord}_w C>1$, as discussed in~\citep[Lemma 14]{bayram2023vector}. However, here we only need to show that it is sufficient.
\end{rem} 


See the Section~\ref{sec:proof} for the proof of Lemma~\ref{lem:P_c_block}. Paraphrasing the statements says that if the matrix $P_C$ for a given cycle $C$ have a permutation block such that ${w}_0$ is not fixed point of $P_C$ and the remaining principal submatrix has a left eigenvector $\tilde{w}$ corresponding to eigenvalue $1$, then the $w$-order of the cycle $C$ is strictly greater than one. Furthermore, by construction, it holds $\tilde{w}M^k=\tilde{w}$ for any $k \in \mathbb{N}$. Therefore, we call the block $M$ {\em invariant blocks of $P_C$}.

Note that if the set $\pi_0$ is empty (by user's decision), then by the definition of subvector, we have $\tilde{w} = w$ and $P_C=M$. We consider only the construction of invariant blocks, which leads to $\operatorname{ord}_w C$ is $1$, (i.e. $wP_C=w)$.

\textbf{Invariant Blocks for $A_e$.} In this section, we provide an algorithm to construct the set of stochastic matrices $\{A_e \mid e \in G\}$ such that each local stochastic matrix $A_e$ ensures that the subvector of $w$ induced by $\pi_a$ is a left eigenvector corresponding to eigenvalue $1$ of the principal block submatrix of each $A_e$ whose entries indexed by the element $\pi_a$ of partition $\pi$ for all $a = 1, \cdots, \ell$. {We define a mapping $$\Psi: (i, k) \mapsto (i-1)m + k$$ that associates the entry $k \in \{1, \dots, m\}$ of the agent $i \in \{1, \dots, n\}$ with the corresponding each index $\ell\in \{1,2,\cdots,nm\}$.} We denote $nm$-by-$nm$ square matrix with $1$ at the $ij$th entry and $0$ elsewhere by $E_{ij}$. We introduce the definition of the rate matrix set:
\begin{defn}[Rate Matrix Set on $w$]\label{def:rms}
Let  $\mathcal{B}^{ij}_{kl}(w)$ be the set of $nm \times nm$ stochastic matrices for the $k$th entry of the agent $i$ and the $l$th entry of the agent $j$ such that, for all \( B \in \mathcal{B}^{ij}_{kl}(w) \), it satisfies  
\begin{align}\label{eqn:rms}\\[-2em]
     B = &(1-\beta_1) E_{\Psi(i,k),\Psi(i,k)} +  \beta_1 E_{\Psi(i,k),\Psi(j,l)} + \beta_2 E_{\Psi(j,l),\Psi(i,k)} \nonumber \\
        &+ (1 - \beta_2) E_{\Psi(j,l),\Psi(j,l)} + \sum_{g \neq \Psi(i,k),\Psi(j,l)} E_{g,g}\nonumber
\end{align}
where the parameters satisfy  ${\beta_1}/{\beta_2} = {\alpha^j_l}/{\alpha^i_k}$ with $\beta_1, \beta_2 \in (0,1).$
\end{defn}

Paraphrasing Definition~\ref{def:rms} says that the principal $2$-by-$2$ submatrix ,corresponding to the $k$th entry of the agent $i$ and the $l$th entry of the agent $j$, of a matrix in $\mathcal{B}^{ij}_{kl}(w)$  is given by 
\begin{align*}\\[-2em]
  \begin{bmatrix}
1-\beta_1 & \beta_1  \\
\beta_2 & 1-\beta_2
\end{bmatrix}  \mbox{ such that } \beta_1/\beta_2 = {\alpha^j_l}/{\alpha^i_k} \\[-2em]
\end{align*}
and the complementary principal submatrix of $B$ is the identity matrix $I_{2m-2}$. 

From \cite[Proposition 6]{chen2022gossip}, we know that for any pair of $\alpha^i_k, \alpha^j_l \in (0, 1)$, the set of rate matrices is non-empty, that is, there exists a rate matrix on every $w \in \operatorname{int} \Delta^{nm-1}$. One can easily see that there are infinitely many tuples $(\beta_1, \beta_2)$ satisfying the condition for fixed $(\alpha^j_l,\alpha^i_k)$.
A stochastic matrix $B$ constructed by~\eqref{eqn:rms} has one degree of freedom due to the constraint $\alpha^i_k \beta_1 = \alpha^j_l \beta_2$. Then, the set $\mathcal{B}^{ij}_{kl}(w)$ forms a one-dimensional manifold (which can be interpreted as a curve in the space of stochastic matrices) in the space of $nm \times nm$ matrices (which is an ambient space of $2nm$ dimensions).

\textbf{Algorithm.} Now, we can present Algorithm~\ref{alg:state_transition}.
\begin{algorithm}
\caption{Construction of a local stochastic matrix \( A_{ij} \) for a given partition ${\pi}$ and a given weight vector $w$}\label{alg:state_transition}
\begin{algorithmic}
\STATE \textbf{Input:} $ {\pi}=\{ \pi_a \}_{a=0}^\ell, w , \{i,j\} \gets v_iv_j$ 
\STATE \textbf{Output:} \( A_{ij} \)
\STATE \textbf{Initialize} \( A_{ij} = I \) 
\STATE $\triangleright  \mbox{Loop for the elements $\pi_a$ of partition }$

\FOR{$a = 1 \text{ to } \ell$}  
    \STATE $\triangleright  \mbox{Loop over the entries of the agent } i$
    \FOR{$k = 1 \text{ to } m$}  
        \STATE $\triangleright  \mbox{Loop over the entries of the agent } j$
        \FOR{$l = 1 \text{ to } m$}
            \IF{$(i-1)m + k  \in\pi_a  \mbox{ and } (j-1)m + l  \in\pi_a $}
                \STATE $ \alpha^j_l \gets w[(j-1)m + l] $ 
                \STATE $ \alpha^i_k \gets w[(i-1)m + k] $ 
    
                \STATE \( r \gets {\alpha^j_l}/{\alpha^i_k} \)
                \STATE Pick $\beta_1$ and $\beta_2$ 
                 such that $\beta_1 / \beta_2 = r$ with $\beta_1,\beta_2 \in (0,1)$
                \STATE Get a matrix $B^{ij}_{kl} \in \mathcal{B}^{ij}_{kl}(w)$ via equation~\eqref{eqn:rms} 
    
                \STATE Update the stochastic matrix:
                \STATE \hspace{1cm} \( A_{ij} \gets A_{ij} \cdot B^{ij}_{kl} \)
            \ENDIF
        \ENDFOR
    \ENDFOR
\ENDFOR
\RETURN $A_{ij}$
\end{algorithmic}
\end{algorithm}

\begin{rem}\label{rem:perm}
Note that the algorithm starts iterating from the element $\pi_1$ (that is, it ignores the element $\pi_0$ of the partition corresponding to the permutation block). Once the user decides to have permutation block at the limit (i.e. non-empty $\pi_0$), it is sufficient to place a (non-identity)  permutation matrix in the principal block submatrix of $A_{ij}$, whose entries are indexed by all $k$ and $l$ such that $\Psi(i,k) \in \pi_0$ and $\Psi(j,l) \in \pi_0$ such that the subvector of $w$ induced by $\pi_0$ is not a fixed point of this submatrix,(see Lemma~\ref{lem:P_c_block}). 
\end{rem}

\begin{thm}[Correctness of Algorithm~\ref{alg:state_transition}]\label{thm:holonomic}
Let $G=(V,E)$ be a simple, $2$-edge connected graph on $n$ nodes with matrix-valued edge weights $A_e$, $e \in E$. Let ${\pi}$ be a admissible partition of index set on $G$. Let $\{ A_e \in \mathbb{R}^{nm \times nm} \mid e \in E \}$ be the set of local stochastic matrices constructed using Algorithm~\ref{alg:state_transition}. Then, the set $\{ A_e \in \mathbb{R}^{nm \times nm} \mid e \in E \}$ is $w$-holonomic for $G$. 
\end{thm}  
See Section~\ref{sec:proof} for the proof of Theorem~\ref{thm:holonomic}. We now provide an example for the algorithm.

\begin{exmp}[Cont.] Recall  the graph $G$ and admissible partition $\pi$ of index set on $G$ in Example~\ref{exmp:derived}. 

Let $w=[0.012, 0.209, 0.062, 0.027 , 0.050,
       0.081, 0.013, 0.544]$ be the weight vector given by the user. Now, consider the edge $e = v_1 v_2$. For $k = 1$, we have the index $1$ (where $\Psi(1,1)=1$), which is not in $\pi_1$. So, the condition in Line $10$ of Algorithm~\ref{alg:state_transition} is not satisfied. We pass to $k = 2$ (where $\Psi(1,2)=2$) and $l = 2$ (where $\Psi(2,2)=4$), we have the following values from the weight vector $w$:
\[
\alpha^1_2 = 0.209, \quad \alpha^2_2 = 0.027.
\]
Then, the rate $r =\alpha^2_2/ \alpha^1_2 = 0.129$. Using this rate, we select $\beta_1 = 0.082$ and $\beta_2 = 0.630$, since $\frac{\beta_1}{\beta_2} = 0.129$. By using equation~\eqref{eqn:rms}, we obtain the matrix $B$ as follows:
\begin{align*}
 \mathcal{B}^{12}_{22}(w) \ni B = 
\left[ \begin{smallmatrix} 
1.000 &   &   &   &   &   &   &   \\ 
  & 0.918 &   & 0.082 &   &   &   &   \\ 
  &   & 1.000 &   &   &   &   &   \\ 
  & 0.630 &   & 0.370 &   &   &   &   \\ 
  &   &   &   & 1.000 &   &   &   \\ 
  &   &   &   &   & 1.000 &   &   \\ 
  &   &   &   &   &   & 1.000 &   \\ 
  &   &   &   &   &   &   & 1.000 \\ 
\end{smallmatrix} \right]
\end{align*}
Note that only the principal $2$-by-$2$ submatrix whose entries are indexed by $\{2,4\}$ is different from the identity matrix. For the edge ${v_1v_2}$, there is no other pair that satisfies the condition in Line $10$. Now, we can then check the permutation block. Since the indices corresponding to the permutation block $\pi_0=\{1,3\}$ belong to agent $1$ and agent $2$ (i.e. $\Psi(1,1) = 1 \in \pi_0$ and $\Psi(2,1) = 3 \in \pi_0$), we can place a non-identity permutation matrix in the principal block submatrix of $B$ whose entries are indexed by $\{1,3\}$ (see Remark~\ref{rem:perm}). Then, we have the following final form of the local stochastic matrix:



\begin{align*}
 A_{12} = 
\left[ \begin{smallmatrix} 
 &   & 1.000  &   &   &   &   &   \\ 
  & 0.918 &   & 0.082 &   &   &   &   \\ 
 1.000 &   &  &   &   &   &   &   \\ 
  & 0.630 &   & 0.370 &   &   &   &   \\ 
  &   &   &   & 1.000 &   &   &   \\ 
  &   &   &   &   & 1.000 &   &   \\ 
  &   &   &   &   &   & 1.000 &   \\ 
  &   &   &   &   &   &   & 1.000 \\ 
\end{smallmatrix} \right]
\end{align*}
\end{exmp}
\section{Proof of Main Results}\label{sec:proof}



\begin{proof}[Proof of Lemma~\ref{lem:P_c_block}].
We assume that $P$ is a permutation matrix and that $w_0 P \neq w_0$. Then, there exists a finite $k > 1$ such that $P^k = I$, which implies that $w_0 P^k = w_0$. Now, we take the $k$-th power of~\eqref{eqn:factorization_Pc} and multiply both sides by the weight vector $w = [ w_0 , \tilde{w} ]$ on the left. Then we have the following:
\begin{align*}\label{eqn:proof_factor}
 \underbrace{[ w_0 , \tilde{w} ]}_{{w}} ({P}_C)^k &= [ w_0 , \tilde{w} ] \left[\begin{smallmatrix}
P  & 0 \\
0 &  M
 \end{smallmatrix}\right]^{k} \\
&= [w_0 (P)^{k}, \underbrace{\tilde{w} (M)^{k}}_{*} ]
\end{align*}
Additionally, we assume that $\tilde{w} M = \tilde{w}$, from which it follows that $*$ in \eqref{eqn:proof_factor} is also equal to $\tilde{w}$. Then, we have the following:
\begin{align*}
    w (P_C)^k = [ w_0 , \tilde{w} ] ({P}_C)^k &= [w_0 , \tilde{w}  ] = w
\end{align*}
From Definition~\ref{def:non-trivial}, $\operatorname{ord}_w C$ is $k$. It completes the proof.
\end{proof}

We need the following lemma:
\begin{lem}\label{lem:right_eigen_B}
For a given weight vector $w \in \operatorname{int}\Delta^{nm-1}$. It holds that
$$
        w B= w , \forall B \in \mathcal{B}^{ij}_{kl}(w)
$$
for any distinct pair $(i,k)$ and $(j,l)$.
\end{lem}

\begin{proof}
From Definition~\ref{def:rms}, only the principal block $2$-by-$2$ submatrix of $B \in \mathcal{B}^{ij}_{kl}(w)$, whose entries are indexed by $\Psi(i,k)$ and $\Psi(j,l)$, and denoted by $\bar{B}^{ij}_{kl}$, is different from the identity matrix. Then, it is sufficient to show that $[ \alpha^i_k , \alpha^j_l]$ (i.e. the subvector induced by the set $\{\Psi(i,k), \Psi(j,l)\}$) is a left eigenvector of the submatrix $\bar{B}^{ij}_{kl}$ corresponding to the eigenvalue $1$.
\begin{align}
&=\begin{bmatrix}
\alpha^i_k & \alpha^j_l  
\end{bmatrix} 
\begin{bmatrix}
1-\beta_1 & \beta_1  \\
\beta_2 & 1-\beta_2 
\end{bmatrix} \\ 
&= \begin{bmatrix}
\alpha^i_k( 1-\beta_1 ) +  \alpha^j_l  \beta_2 , & \alpha^i_k \beta_1 +  \alpha^j_l  (1-\beta_2)  \\
\end{bmatrix} \\
&= \begin{bmatrix}
\alpha^i_k - \alpha^i_k \beta_1 +  \alpha^j_l   \beta_2 , & \alpha^i_k  \beta_1 +  \alpha^j_l - \alpha^j_l  \beta_2 
\end{bmatrix} = \begin{bmatrix}
\alpha^i_k & \alpha^j_l 
\end{bmatrix} 
\end{align} 
where it holds that $\alpha^i_k \beta_1 =  \alpha^j_l   \beta_2$ from Definition~\ref{def:rms}. It completes the proof.
\end{proof}
Paraphrasing, the statement says that the vector $w$ is a left eigenvector of any matrix $B$ in the rate matrix set on $w$, $\mathcal{B}^{ij}_{kl}(w)$, for any given pair $(i,k)$ and $(j,l)$, corresponding to the eigenvalue $1$.


\begin{proof}[Proof of Theorem~\ref{thm:holonomic}]
First, we show that the weight vector $w$ is a left eigenvector corresponding to eigenvalue $1$ of the matrix $A_e$ for any $e \in E$, constructed using Algorithm~\ref{alg:state_transition}. We have an admissible partition $\pi$ of the index set on $G$. From \ref{def:derived_graph_w}, the set of partitions either corresponds to the same agent or is connected through existing edges for gossiping. This implies that for each edge, there is at least one pair of indices that satisfies the condition stated in Line 10 of Algorithm~\ref{alg:state_transition}. Let $r$ be the total number of pairs of indices that satisfy the condition specified in Line 10. Then, we have the following:
\begin{align}\label{eqn:A_e_open}
        A_e = B_1 B_2 \dots B_r
\end{align}
where each $B_a$ is constructed using Equation~\eqref{eqn:rms} for the corresponding indices. This implies that each matrix $B_a$ belongs to the rate matrix set on $w$ for the corresponding indices, (i.e. $B_a \in \mathcal{B}^{ij}_{kl}(w)$). From Lemma~\ref{lem:right_eigen_B}, we know that $w B_a = w$ for any $a = 1, 2, \dots, r$. If we multiply both sides of~\eqref{eqn:A_e_open} by $w$ on the left, we get the following:
\begin{align}
        w A_e &=  w B_1 B_2 \cdots B_r \\
              &= w B_2 \cdots B_r \\
              & \vdots \\ 
              &= w 
\end{align}
Now, consider an arbitrary cycle $C=e_1 \cdots e_{b-1} e_b$, then, we have the following state transition matrix $P_C$: 
\begin{align}
     P_C &= A_{e_b}  A_{e_{b-1}} \cdots A_{e_{1}} 
\end{align}
If we multiply both side by $w$ on the left, then we have:
\begin{align}
     w P_C &= {w A_{e_b} } A_{e_{b-1}} \cdots A_{e_{1}}
\end{align}
Each matrix $A_e$ for $e \in C$ is constructed using Algorithm~\ref{alg:state_transition}. Thus, we have the following iteration: 
\begin{align}
     w P_C &= \underbrace{w A_{e_b} }_{w} A_{e_{b-1}} \cdots A_{e_{1}} \\
    &= \underbrace{w A_{e_{b-1}}}_{w} \cdots A_{e_1} \\
     &= w
\end{align}
From~\cite[Theorem 3.3]{merris2011graph}, we know that every node in a $2$-edge connected, simple graph $G$ is covered by at least one cycle. This shows that the repetition algorithm for all $e \in E$ leads to covering all cycles. It implies that $w P_C = w$ for all cycles $C$ in $\vec{\mathcal{C}}$. From Definition~\ref{def:non-trivial}, we show that the set of local stochastic matrices $\{ A_e, e \in E\}$ is $w$-holonomic for $G$. It completes the proof.  \end{proof}





\section{Conclusion}


In this work, we have proposed an algorithm to construct a set of local stochastic matrices governing the gossip process, ensuring convergence to multiple consensus with a given weight vector and consensus cluster partition using the concept of holonomy. Our method is based on the construction of stochastic matrices with a known left eigenvector. Our contributions offer a solution to the open problem of realizing gossip matrices for distributed control systems, enabling scalable consensus protocols. This algorithm can be applied to federated learning or other decentralized optimization problems, where allowing agents to collaborate on optimization tasks while reaching consensus based on predefined weights at the limit.

\bibliographystyle{agsm}        

\bibliography{holonomy}

\end{document}